\documentclass[11pt]{article}
\usepackage{lmodern}

\usepackage[utf8]{inputenc}

\usepackage{geometry}%
\usepackage{graphicx}
\usepackage{amsmath, amsthm, amsfonts, amssymb, graphicx, cite, epsfig,epstopdf,color,soul}
\usepackage[bookmarks=true,pageanchor,colorlinks,linkcolor=red,anchorcolor=blue, citecolor=blue,urlcolor=blue,hyperfootnotes=false]{hyperref}

\newcommand{\Eset}{\mathbb{E}}

\newcommand{\Pset}{\mathbb{P}}

\newcommand{\Rset}{\mathbb{R}}

\newcommand{\Fcal}{{\cal F}}
\newcommand{\Gcal}{{\cal G}}

\newcommand{\Kcal}{{\cal K}}

\newcommand{\Ocal}{{\cal O}}

\newcommand{\Scal}{{\cal S}}

\newcommand{\Xcal}{{\cal X}}

\newcommand{\xhat}{{\hat{x}}}

\newtheorem{lem}{Lemma}
\newtheorem{thm}{Theorem}

\newtheorem{assump}{Assumption}

\title{Finite-Time Analysis of Stochastic Gradient Descent under Markov Randomness}
\author{
Thinh T. Doan \qquad Lam M. Nguyen
\qquad Nhan H. Pham \qquad
Justin Romberg\thanks{Thinh T. Doan and Justin Romberg are with the School of Electrical and Computer Engineering, Georgia Institute of Technology, GA, 30332, USA. Email: \{{\tt\small thinhdoan@gatech.edu, jrom@ece.gatech.edu}\}. Lam M. Nguyen is with IBM Research, Thomas J. Watson Research Center, Yorktown Heights, NY, USA. Email:  {\tt\small lamnguyen.mltd@ibm.com}. Nhan H. Pham is with the Department of Statistics and Operations Research, University of North Carolina at Chapel Hill, Chapel Hill, NC, USA. Email: {\tt\small nhanph@live.unc.edu}.}
}
\date{}

\begin{document}

\maketitle

\begin{abstract}
Motivated by broad applications in reinforcement learning and machine learning, this paper considers the popular stochastic gradient descent ({\sf SGD}) when the gradients of the underlying objective function are sampled from Markov processes.  This Markov sampling leads to the gradient samples being biased and not independent. The existing results for the convergence of {\sf SGD} under Markov randomness are often established under the assumptions on the boundedness of either the iterates or the gradient samples.  Our main focus is to study the finite-time convergence of {\sf SGD} for different types of the objective functions, without requiring these assumptions. We show that SGD converges nearly at the same rate with Markovian gradient samples as with independent gradient samples.  The only difference is a logarithmic factor that accounts for the mixing time of the Markov chain. 
\end{abstract}


\section{Introduction}\label{sec:intro}
Stochastic gradient descent ({\sf SGD}), originally introduced in \cite{RM1951} under the name of stochastic approximation {\sf SA}, is probably the most efficient and powerful method for solving optimization problems in machine learning. In this context, we want to optimize an (unknown) objective function $f$ defined over a large set of collected data, while we only have an access to the noisy samples of the gradient $\nabla f$ of $f$. At a point $x$, we observe a random vector $G(x, \xi)$, a sample of $\nabla f(x)$, where $\xi$ is some random variable.  Through a careful choice of step sizes, the ``noise'' induced by this randomness can be averaged out across iterations, and the algorithm converges to stationary point of $f$. Such convergence is established based on two standard assumptions, the gradient samples are i.i.d and unbiased, i.e., $\Eset[G(x,\xi)] = \nabla f(x)$ \cite{GhadimiL2013,bottou2016optimization,Nguyen2018_sgdhogwild,Nguyen2019_sgd_new_aspects}. For more details about the results of {\sf SGD} under this setting can be found in the recent review paper \cite{bottou2016optimization}.

In this paper, we show that a particular version of {\sf SGD} is still ergodic when the gradients of the underlying objective function are sampled from Markov processes, and hence are biased and not independent across iterations. It has been observed that the {\sf SGD} performs better when the gradients are sampled from Markov process as compared to i.i.d samples in both convex and nonconvex problems \cite{SunSY2018}. This model for the gradients has been considered previously in \cite{DuchiAJJ2012, SunSY2018,JohanssonRJ2010,RamNV2009,DoanNPR2020}, where different variants of {\sf SGD} are considered. In particular, the Markov subgradient incremental methods have been considered in \cite{RamNV2009,JohanssonRJ2010} motivated by the applications of distributed algorithms to minimize a sum of functions over a network of agents. At any iteration, an agent updates the current iterates using its own function, and either passes the new iterate to a randomly chosen neighbor or keeps updating. In \cite{DuchiAJJ2012}, the authors consider an ergodic version of the popular mirror descent and provide a finite-time analysis for the case of convex optimization problems. The authors in \cite{SunSY2018} study a non-ergodic version of {\sf SGD} for both convex and nonconvex problems when the Markov chains is non-reversible. They establish a convergence rate $\Ocal(1/k^{1-q}), q\in(1/2,1)$, which is worse than $\Ocal(1/\sqrt{k})$ under i.i.d sampling. Finally, the authors in \cite{DoanNPR2020} consider an accelerated ergodic version of {\sf SGD} for both nonconvex and convex problems and achieve the same rates as the ones using i.i.d samples, except for a log factor capturing the mixing time of the underlying Markov chain.         

The convergence analysis of the existing works listed above is established under the assumptions on the boundedness of the iterates, e.g., either the feasible set is compact or the gradient is bounded. Such assumptions may not be applicable in many settings, for example, in the context of reinforcement learning ({\sf RL}), which is the main motivation of this paper. Thus, our goal in this paper is to study the convergence rates of {\sf SGD} under Markov randomness for different settings under less restrictive assumptions as compared to the existing works \cite{DuchiAJJ2012, SunSY2018,JohanssonRJ2010,RamNV2009,DoanNPR2020}. We also note some relevant works in \cite{SrikantY2019_FiniteTD,ChenZDMC2019,Karimi_colt2019}, where the authors study finite-time analysis of stochastic approximation under Markov randomness. We, however, focus on the convergence rates of {\sf SGD} under different settings than the ones considered in \cite{SrikantY2019_FiniteTD,ChenZDMC2019,Karimi_colt2019}.    


\subsection{Main contributions}
We study the popular stochastic gradient descent where the gradients are sampled from
a Markov process. We show that, despite the gradients being biased and dependent across iterations, the convergence rate across many different types of objective functions (strongly convex, smooth and satisfying error bound conditions,
nonconvex and smooth) is within a logarithmic factor of the comparable bounds for independent gradients. This logarithmic factor is naturally related to the mixing time of the underlying Markov process generating the stochastic gradients. Our convergence analysis is established without requiring any assumption on the boundedness of the iterates or the gradients, as often assumed in the existing literature.

\section{Stochastic gradient descent under Markov randomness}
In this section, we present the main problem considered in this paper, which is inspired by the work in \cite{DuchiAJJ2012}. Specifically, we consider the following (possibly nonconvex) optimization problem
\begin{align}
\underset{x\in\Rset^{d}}{\text{minimize}}\; f(x) ,\label{sec_optimization:prob}
\end{align}
where $f:\Rset^{d}\rightarrow\Rset$ is given as
\begin{align}
f(x) \triangleq \Eset_{\mu}[\Gcal(x;\xi)] = \int_{\Scal}\Gcal(x;\xi)d\mu(\xi).\label{sec_optimization:notation_f}
\end{align}
Here $\Scal$ is a finite statistical sample space with probability distribution $\mu$ and $\Gcal(\cdot,\xi):\Scal\rightarrow\Rset$ is a bounded below (possibly nonconvex) function  associated with $\xi\in\Scal$. We are interested in the first-order stochastic optimization methods, specifically the popular {\sf SGD}, for solving problem \eqref{sec_optimization:prob}. In particular, we assume that we only have access to $G(x;\xi)\in\partial \Gcal(x;\xi)$, the (sub)gradient of $\Gcal(\cdot;\xi)$ evaluated at the point $x$. Starting with an arbitrarily initial condition $x_{0}\in\Xcal$, {\sf SGD} iteratively updates the iterates as
\begin{align}
x_{k+1} = x_{k} - \alpha_{k} G(x_{k},\xi_{k}), \label{sec_SGD:alg}
\end{align}
where $\alpha_{k}$ is some nonincreasing and nonnegative time-varying step sizes. We consider in this paper the case where the samples $\xi$ are generated from an ergodic Markov process, whose stationary distribution is $\mu$. Our focus is to study the finite-time analysis of {\sf SGD} under such Markov randomness on different assumptions of the objective function $f$. To do that, we introduce the following technical assumptions used throughout this paper.   

\subsection{Assumptions and Preliminaries}
Given a vector $x\in\Rset^{d}$, we associate with it a norm $\|\cdot\|$ and denote its dual norm as $\|\cdot\|_{*}$, i.e., 
\begin{align*}
\|x\|_{*} = \sup_{\|y\|\leq 1}\langle x,y\rangle,
\end{align*}
where $\langle \cdot,\cdot\rangle$ denotes the inner product. Using the Cauchy-Schwarz inequality we obtain
\begin{align}
|\langle x,y\rangle| \leq \|x\|_{*}\|y\|.     
\end{align}
Moreover, by \eqref{sec_SGD:alg} we have
\begin{align}
\|x_{k+1}-x_{k}\| \leq \alpha_{k}\|G(x_{k},\xi_{k})\|_{*}.\label{sec_SGD_sc:lem_xk_ineq}    
\end{align}
We first assume that problem \eqref{sec_optimization:prob} is well-defined. 
\begin{assump}\label{assump:optimal_set}
The optimal set $\Xcal^*$ of \eqref{sec_optimization:prob} is nonempty, i.e., there exists $x^*$ s.t. $f(x^*) = \min_{x} f(x)$. 
\end{assump}
We consider the following fairly standard technical assumptions about the Markov process, which are often assumed in the existing literature  \cite{DuchiAJJ2012, SunSY2018, SrikantY2019_FiniteTD,ChenZDMC2019, Doan2019, DoanNPR2020, DoanMR2019_DTD}.   

\begin{assump}\label{assump:stationary}
The sequence $\{\xi_{k}\}$ is a Markov chain with a finite state space $\Scal$. In addition, the following limits exit
\begin{align}
\lim_{k\rightarrow\infty}
\Eset[G(x;\xi_{k})] = g(x)\in\partial f(x)\quad \forall x.
\end{align}
\end{assump}
\begin{assump}\label{assump:mix}
Given a positive constant $\alpha$, we denote by $\tau(\alpha)$ the mixing time of the Markov chain $\{\xi_{k}\}$. We assume that for $g(x) \in \partial f(x)$ and $\forall\;\xi\in\Scal$
\begin{align*}
&\left\|\Eset[G(x;\xi_{k})] - g(x)\,|\, \xi_{0} = \xi\right\|_{*} \leq \alpha,\quad \forall  k\geq \tau(\alpha).
\end{align*}
Moreover, $\{\xi_{k}\}$ has a geometric mixing time, i.e., there exists a positive constant $C$ such that \begin{align}
\tau(\alpha) = C\log\left(\frac{1}{\alpha}\right).\label{notation:tau}
\end{align} 
\end{assump} 
Next we assume that $G(x,\cdot)$ is uniformly Lipschitz continuous, which is satisfied in the context of {\sf RL} studied in \cite{ChenZDMC2019,SrikantY2019_FiniteTD}. Moreover, this condition is weaker than the boundedness assumptions in \cite{DuchiAJJ2012, SunSY2018,JohanssonRJ2010,RamNV2009,DoanNPR2020}.

\begin{assump}\label{assump:Lipschitz_sample}
There exists a constant $M$ such that for all $\xi\in\Scal$
\begin{align}
\|G(x,\xi)\|_{*}   \leq M (\|x\|+1)\qquad \text{and}\qquad \|g(x)\|_{*}   \leq M (\|x\|+1),\qquad \forall x\in\Xcal, g(x)\in\partial f(x).  \label{assump:Lipschitz_sample:ineq}
\end{align}
\end{assump}
For convenience, we introduce the following notation
\begin{align}
\alpha_{k;\tau(\alpha_{k})} =     \sum_{u = k-\tau(\alpha_{k})}^{k-1}\alpha_{u},\label{sec_SGD_sc:notation_alpha_tau}
\end{align}
where $\tau(\alpha_{k})$ is the mixing time defined in Assumption \ref{assump:mix} associated with the step size $\alpha_{k}$. Note that $\tau(\alpha_{k}) = \log(1/\alpha_{k})$ . Moreover, we consider in this paper diminishing step size $\alpha_{k}$, therefore, $\lim_{k\rightarrow\infty}\tau(\alpha_{k})\alpha_{k} = 0$. This implies that there exists a positive integer $\Kcal^*$ such that
\begin{align}
\alpha_{k;\tau(\alpha_{k})}\leq \tau(\alpha_{k})\alpha_{k-\tau(\alpha_{k})}\leq \frac{\log(2)}{M},\qquad \forall k \geq \Kcal^*.  \label{sec_SGD:notation_K*}
\end{align} 
Finally, we consider the following important lemma to provide an upper bound of the iterates generated by {\sf SGD}. This result is used throughout this paper to derive the rates of {\sf SGD}.
\begin{lem}\label{sec_SGD_sc:lem_xk_bounded}
Suppose that Assumptions \ref{assump:stationary}--\ref{assump:Lipschitz_sample} hold. Let $\{x_{k}\}$ be generated by \eqref{sec_SGD:alg}. Then for all $k\geq \Kcal^*$
\begin{align}
&\|x_{k}-x_{k-\tau(\alpha_{k})}\|\leq 6M\alpha_{k;\tau(\alpha_{k})}\|x_{k}\| + 6M\alpha_{k;\tau(\alpha_{k})}\leq 2\|x_{k}\| + 2.   \label{sec_SGD_sc:lem_xk_bounded_ineq2}
\end{align}
\end{lem}

\begin{proof}
Using \eqref{assump:Lipschitz_sample:ineq}, \eqref{sec_SGD_sc:lem_xk_ineq}, and the triangle inequality we have
\begin{align}
\|x_{k+1}\| - \|x_{k}\| \leq \|x_{k+1} - x_{k}\| \leq \alpha_{k}\|G(x_{k},\xi_{k})\|_{*} \leq M\alpha_{k}(\|x_{k}\|+1),\label{sec_SGD_sc:lem_xk_bounded_Eq1a}    
\end{align}
which gives
\begin{align*}
\|x_{k+1}\| \leq (1+M\alpha_{k})\|x_{k}\| + M\alpha_{k}.
\end{align*}
Using the relation $1+x \leq e^{x}$ for all $x\geq 0$, the preceding relation gives for all $t\in[k-\tau(\alpha_{k}),k-1]$
\begin{align*}
\|x_{t}\| &\leq \prod_{u = k-\tau(\alpha_{k})}^{t-1}(1+M\alpha_{u})\|x_{k-\tau(\alpha_{k})}\| + M\sum_{u = k-\tau(\alpha_{k})}^{t-1}\alpha_{u}\prod_{\ell=u+1}^{t-1}(1 + M\alpha_{\ell})\notag\\
&\leq \exp\left\{M\sum_{u=k-\tau(\alpha_{k})}^{t-1}\alpha_{u}\right\}\|x_{k-\tau(\alpha_{k})}\| + M\sum_{u = k-\tau(\alpha_{k})}^{t-1}\alpha_{u}\exp\left\{M\sum_{\ell = u+1}^{t-1}\alpha_{\ell}\right\}\notag\\
&\leq \exp\left\{M\tau(\alpha_{k})\alpha_{k-\tau(\alpha_{k})}\right\}\|x_{k-\tau(\alpha_{k})}\| + M\sum_{u = k-\tau(\alpha_{k})}^{t-1}\alpha_{u}\exp\left\{M\tau(\alpha_{k})\alpha_{k-\tau(\alpha_{k})}\right\}.
\end{align*}
Recall from \eqref{sec_SGD:notation_K*} that for all $k\geq \Kcal^*$
\begin{align*}
M\tau(\alpha_{k})\alpha_{k-\tau(\alpha_{k})}\leq \log(2) \leq \frac{1}{2}\cdot
\end{align*}
Thus, by using \eqref{sec_SGD_sc:notation_alpha_tau} we have from the equation above for all $t\in[k-\tau(\alpha_{k}),k-1]$
\begin{align*}
\|x_{t}\| &\leq  2\|x_{k-\tau(\alpha_{k})}\| + 2M\sum_{u = k-\tau(\alpha_{k})}^{t-1}\alpha_{u} = 2\|x_{k-\tau(\alpha_{k})}\| + 2M\alpha_{k;\tau(\alpha_{k})}.
\end{align*}
Using the preceding relation we have from \eqref{sec_SGD_sc:lem_xk_bounded_Eq1a} for all $k\geq \Kcal^*$
\begin{align*}
\|x_{k}-x_{k-\tau(\alpha_{k})}\|&\leq \sum_{t=k-\tau(\alpha_{k})}^{k-1}\|x_{t+1} - x_{t}\|\leq \sum_{t=k-\tau(\alpha_{k})}^{k-1}M\alpha_{t}(\|x_{t}\| + 1)\notag\\
&\leq \sum_{t=k-\tau(\alpha_{k})}^{k-1}L\alpha_{t}\left(2\|x_{k-\tau(\alpha_{k})}\| + 2M\alpha_{k;\tau(\alpha_{k})}\right) + M\alpha_{k;\tau(\alpha_{k})}\notag\\
&\leq 2M\alpha_{k;\tau(\alpha_{k})}\|x_{k-\tau(\alpha_{k})}\| + 2M\alpha_{k;\tau(\alpha_{k})},
\end{align*}
where the last inequality is due to \eqref{sec_SGD:notation_K*}. Using the preceding inequality and the triangle inequality yields
\begin{align*}
\|x_{k}-x_{k-\tau(\alpha_{k})}\|&\leq 2M\alpha_{k;\tau(\alpha_{k})}\|x_{k} - x_{k-\tau(\alpha_{k})}\| + 2M\alpha_{k;\tau(\alpha_{k})}\|x_{k}\| + 2M\alpha_{k;\tau(\alpha_{k})}\notag\\
&\leq \frac{2}{3}\|x_{k} - x_{k-\tau(\alpha_{k})}\| + 2M\alpha_{k;\tau(\alpha_{k})}\|x_{k}\| + 2M\alpha_{k;\tau(\alpha_{k})},
\end{align*}
where the last inequality we use \eqref{sec_SGD:notation_K*} to  have $L\alpha_{k;\tau(\alpha_{k})}\leq \log(2)\leq 1/3$. Rearranging the equation above yields \eqref{sec_SGD_sc:lem_xk_bounded_ineq2}
\begin{align*}
 \|x_{k}-x_{k-\tau(\alpha_{k})}\| \leq  6M\alpha_{k;\tau(\alpha_{k})}\|x_{k}\| + 6M\alpha_{k;\tau(\alpha_{k})}\leq 2\|x_{k}\| + 2.  
\end{align*}
\end{proof}

\section{Finite-time analysis of {\sf SGD} under Markov randomness}\label{sec:SGD:analysis}
In this section, we focus on analyzing the finite-time performance of {\sf SGD} under Markov randomness for different assumptions on the objective function $f$. In particular, we consider three different cases, namely, $f$ is strongly convex, $f$ is smooth and satisfies an error bound condition, and $f$ is only smooth and nonconvex. A key observation of our results is that {\sf SGD} under Markov randomness has the same finite-time convergence as the one under i.i.d samples and unbiased estimate, except for a $\log$ factor capturing the impact of geometric mixing time of the underlying Markov processes. More details are given as below.    

\subsection{Convergence analysis: strong convexity}
We study here the convergence rates of {\sf SGD} \eqref{sec_SGD:alg} when the function $f$ is strongly convex, i.e., we consider the following assumption. 
\begin{assump}\label{assump:sc}
There exists a constant $\sigma > 0$ such that $f$ satisfies
\begin{align}
f(x) - f(y) - \langle g(y),x-y\rangle \geq \frac{\sigma}{2}\|x-y\|^2,\qquad g(y) \in\partial f(y).     \label{assump_sc:ineq}
\end{align}
\end{assump}
Under this assumption, there exists a unique optimal solution $x^*$ of \eqref{sec_optimization:prob}, i.e., $\Xcal^* = {x^*}$. In this section, we consider the norm induced by the inner product, so, $\|\cdot\|$ and $\|\cdot\|_{*}$ are the same.  Recall that $\Kcal^*$ is a positive integer satisfying \eqref{sec_SGD:notation_K*}. With some abuse of notation, we consider in this section  $\Kcal^*$ satisfying
\begin{align}
\left\{\begin{array}{l}
\tau(\alpha_{k})\alpha_{k-\tau(\alpha_{k})}\leq \max\left\{\frac{\log(2)}{M},\;\frac{\sigma}{8(25M^2+1)}\right\}\\
\ln\left(\frac{\sigma k}{4}\right)\leq \frac{k}{2}
\end{array}\right.\qquad \forall k \geq \Kcal^*,   \label{sec_SGD_sc:notation_K*}
\end{align} 
which also satisfies \eqref{sec_SGD:notation_K*}. Our main result in this section is to show that $x_{k}$ converges to $x^*$ in expectation at a rate
\begin{align}
\Eset\left[\|x_{k}-x^*\|^2\right]  \lesssim \Ocal\left(\frac{\Eset\left[\|x_{\Kcal^*}-x^*\|^2\right]}{k^2} + \frac{\log(k)}{k}\right).  
\end{align}
To derive this rate, we  consider the following two key lemmas. The first lemma is to provide an upper bound on the function value. 
\begin{lem}\label{sec_SGD_sc:lem_f_fk}
Suppose that Assumptions \ref{assump:sc} holds. Let $\{x_{k}\}$ be generated by \eqref{sec_SGD:alg}. Then we have for any $x$ 
\begin{align}
f(x_{k}) - f(x) &\leq \frac{1-\sigma\alpha_{k}}{2\alpha_{k}}\|x-x_{k}\|^2 - \frac{1}{2\alpha_{k}} \|x-x_{k+1}\|^2 + L^2\alpha_{k}(\|x_{k}\|^2+1)\notag\\ 
&\qquad  
-  \langle G(x_k;\xi_{k}) - g(x_{k}), x_{k} -  x\rangle \rangle.\label{sec_SGD_sc:lem_f_fk_bounded:ineq}
\end{align}
\end{lem}

\begin{proof}
By the strong convexity of $f$, Eq.\ \eqref{assump_sc:ineq}, we have for all $x\in\Rset^{d}$ and $g(x_{k})\in\partial f(x_{k})$
\begin{align}
f(x_{k}) - f(x) &\leq  g(x_k)^T( x_{k} -  x) - \frac{\sigma}{2}\|x_{k}-x\|^2 \notag\\
&=  \langle G(x_k;\xi_{k}), x_{k+1} -  x\rangle + \langle G(x_k;\xi_{k}), x_{k} -  x_{k+1}\rangle\notag\\  
&\qquad - \langle G(x_k;\xi_{k}) - g(x_{k}), x_{k} -  x \rangle - \frac{\sigma}{2}\|x_{k}-x\|^2.\label{sec_SGD_sc:lem_f_fk:Eq1a}
\end{align}
By \eqref{sec_SGD:alg} we have
\begin{align*}
\alpha_{k}\langle G(x_{k},\xi_{k}),x_{k+1}-x\rangle =    \langle x_{k}-x_{k+1}, x_{k+1} - x \rangle = \frac{1}{2}\left[\|x-x_{k}\|^2 - \|x-x_{k+1}\|^2 - \|x_{k}-x_{k+1}\|^2\right].
\end{align*}
Substituting the preceding relation into \eqref{sec_SGD_sc:lem_f_fk:Eq1a} yields \eqref{sec_SGD_sc:lem_f_fk_bounded:ineq}, i.e.,
\begin{align*}
f(x_{k}) - f(x) &\leq \frac{1-\sigma\alpha_{k}}{2\alpha_{k}}\|x-x_{k}\|^2 - \frac{1}{2\alpha_{k}} \|x-x_{k+1}\|^2 - \frac{1}{2\alpha_{k}} \|x_{k}-x_{k+1}\|^2\notag\\ &\qquad + \langle G(x_k;\xi_{k}), x_{k} -  x_{k+1} \rangle  - \langle G(x_k;\xi_{k}) - g(x_{k}), x_{k} -  x \rangle \notag\\
&\leq \frac{1-\sigma\alpha_{k}}{2\alpha_{k}}\|x-x_{k}\|^2 - \frac{1}{2\alpha_{k}} \|x-x_{k+1}\|^2 + \frac{\alpha_{k}\|G(x_{k};\xi_{k})\|_{*}^2}{2}\notag\\ 
&\qquad -  \langle G(x_k;\xi_{k}) - g(x_{k}), x_{k} -  x \rangle\notag\\
&\leq \frac{1-\sigma\alpha_{k}}{2\alpha_{k}}\|x-x_{k}\|^2 - \frac{1}{2\alpha_{k}} \|x-x_{k+1}\|^2 + L^2\alpha_{k}(\|x_{k}\|^2+1)\notag\\ 
&\qquad  
-  \langle G(x_k;\xi_{k}) - g(x_{k}), x_{k} -  x \rangle,
\end{align*}
where the second and last inequalities are due to the Cauchy-Schwarz inequality and \eqref{assump:Lipschitz_sample:ineq}, respectively.
\end{proof}
The next lemma provides an upper bound for the last term on the right-hand side of \eqref{sec_SGD_sc:lem_f_fk_bounded:ineq}. 
\begin{lem}\label{sec_SGD_sc:lem_bias}
Suppose that Assumptions \ref{assump:stationary}--\ref{assump:Lipschitz_sample} hold. Let $\{x_{k}\}$ be generated by \eqref{sec_SGD:alg}. Then $\forall k\geq \Kcal^*$ and $x\in\Rset^{n}$ we have
\begin{align}
\Eset\left[\langle G(x_k;\xi_{k}) - g(x_{k}), x_{k} -  x \rangle\right] \leq \left(\alpha_{k}+24M^2\alpha_{k;\tau(\alpha_{k})}\right)\Eset\left[\|x_{k}\|^2\right] + 3\alpha_{k} + 24M^2\alpha_{k;\tau(\alpha_{k})}.     \label{sec_SGD_sc:lem_bias:ineq}
\end{align}
\end{lem}

\begin{proof}
Consider 
\begin{align}
\langle G(x_k;\xi_{k}) - g(x_{k}), x_{k} -  x \rangle =  \langle G(x_k;\xi_{k}) - g(x_{k}), x_{k} -  x_{k-\tau(\alpha_{k})} \rangle + \langle G(x_k;\xi_{k}) - g(x_{k}), x_{k-\tau(\alpha_{k})} -  x \rangle.
\label{sec_SGD_sc:lem_bias:Eq1}
\end{align}
We next analyze each term on the right-hand side of \eqref{sec_SGD_sc:lem_bias:Eq1a}. Let $\Fcal_{k}$ be the filtration containing all the information generated by {\sf SGD} up to iteration $k$. First, by Assumption \eqref{assump:mix} we have
\begin{align}
&\Eset\left[\langle G(x_k;\xi_{k}) - g(x_{k}), x_{k-\tau(\alpha_{k})} -  x \rangle\,|\,\Fcal_{k-\tau(\alpha_{k})}\right] = \left\langle \Eset[G(x_k;\xi_{k}) - g(x_{k})\,|\,\Fcal_{k-\tau(\alpha_{k})}], x_{k-\tau(\alpha_{k})} -  x \right\rangle\notag\\
&\leq \|x_{k-\tau(\alpha_{k})} -  x\|\left|\Eset[G(x_k;\xi_{k}) - g(x_{k})\,|\,\Fcal_{k-\tau(\alpha_{k})}]\right|\leq \alpha_{k}\|x_{k-\tau(\alpha_{k})} -  x\| \stackrel{\eqref{sec_SGD_sc:lem_xk_bounded_ineq2}}{\leq} 2\alpha_{k} \|x_{k}\| + 2\alpha_{k}\notag\\
&\leq \alpha_{k}\|x_{k}\|^2 + 3\alpha_{k},   \label{sec_SGD_sc:lem_bias:Eq1a}
\end{align}
where the third inequality we use \eqref{sec_SGD_sc:notation_K*} to have $\alpha_{k;\tau(\alpha_{k})}\leq \log(2)/M\leq 1/3$ and the last inequality is due to Cauchy-Schwarz inequality. Second, using \eqref{assump:Lipschitz_sample:ineq} we consider
\begin{align}
&\langle G(x_k;\xi_{k}) - g(x_{k}), x_{k} -  x_{k-\tau(\alpha_{k})} \rangle \leq \left\|\langle G(x_k;\xi_{k}) - g(x_{k}), x_{k} -  x_{k-\tau(\alpha_{k})} \rangle\right\|\notag\\
&\leq \left(\|G(x_k;\xi_{k})\| + \|g(x_{k})\|\right)\|x_{k} -  x_{k-\tau(\alpha_{k})}\| \leq 2M(\|x_{k}\|+1)\|x_{k} -  x_{k-\tau(\alpha_{k})}\|\notag\\
&\stackrel{\eqref{sec_SGD_sc:lem_xk_bounded_ineq2}}{\leq} 12M^2\alpha_{k;\tau(\alpha_{k})} (\|x_{k}\|+1)^2\leq 24M^2\alpha_{k;\tau(\alpha_{k})}\|x_{k}\|^2 + 24M^2\alpha_{k;\tau(\alpha_{k})}.  \label{sec_SGD_sc:lem_bias:Eq1b}
\end{align}
Taking the expectation on both sides of \eqref{sec_SGD_sc:lem_bias:Eq1} and using \eqref{sec_SGD_sc:lem_bias:Eq1a} and \eqref{sec_SGD_sc:lem_bias:Eq1b} we obtain \eqref{sec_SGD_sc:lem_bias:ineq}.
\end{proof}
We now state the main result of this section in the following theorem, where we study the convergence rate of {\sf SGD} when $f$ is strongly convex. 
\begin{thm}\label{sec_SGD_sc:thm_rate}
Suppose that Assumptions \ref{assump:stationary}--\ref{assump:sc} hold. Let $\{x_{k}\}$ be generated by \eqref{sec_SGD:alg}. Let $\alpha_{k}$ be
\begin{align}
\alpha_{k} = \frac{4}{\sigma k}.    \label{sec_SGD_sc:thm_rate:stepsize}
\end{align}
Then we have for all $k\geq \Kcal^*$, where $\Kcal^*$ is defined in \eqref{sec_SGD_sc:notation_K*}
\begin{align}
\Eset[\|x_{k+1}-x^*\|^2]\leq \frac{\Kcal^{*}(\Kcal^{*}-2)}{k^2}\Eset[\|x_{\Kcal^{*}}-x^*\|^2] + \frac{320(15M^2+1)\log(\sigma k/4)}{k\sigma}\cdot    \label{sec_SGD_sc:thm_rate:ineq}
\end{align}
\end{thm}
\begin{proof}
Taking the expectation on both sides of  \eqref{sec_SGD_sc:lem_f_fk_bounded:ineq} and using \eqref{sec_SGD_sc:lem_bias:ineq} we have for $x=x^*$, 
\begin{align*}
\Eset[f(x_{k})] - f(x^*) &\leq \frac{1-\sigma\alpha_{k}}{2\alpha_{k}}\Eset[\|x_{k}-x^*\|^2] - \frac{1}{2\alpha_{k}}\Eset[\|x_{k+1}-x^*\|^2] + M^2\alpha_{k}(\Eset\left[\|x_{k}\|^2\right]+1)\notag\\ 
&\qquad +  \left(\alpha_{k}+24M^2\alpha_{k;\tau(\alpha_{k})}\right)\Eset\left[\|x_{k}\|^2\right] + 3\alpha_{k} + 24M^2\alpha_{k;\tau(\alpha_{k})}\notag\\
&\leq \frac{1-\sigma\alpha_{k}}{2\alpha_{k}}\Eset[\|x_{k}-x^*\|^2] - \frac{1}{2\alpha_{k}}\Eset[\|x_{k+1}-x^*\|^2] \notag\\
&\qquad + (25M^2+1)\alpha_{k;\tau(\alpha_{k})}\Eset\left[\|x_{k}\|^2\right] + (3+25M^2)\alpha_{k;\tau(\alpha_{k})}\notag\\
&\leq \frac{1-\sigma\alpha_{k}}{2\alpha_{k}}\Eset[\|x_{k}-x^*\|^2] - \frac{1}{2\alpha_{k}}\Eset[\|x_{k+1}-x^*\|^2]\notag\\
&\qquad   + (50M^2+2)\alpha_{k;\tau(\alpha_{k})}\Eset\left[\|x_{k}-x^*\|^2\right] + 5(1+15M^2)\alpha_{k;\tau(\alpha_{k})},
\end{align*}
where the second inequality we use $\alpha_{k}\leq \alpha_{k;\tau(\alpha_{k})}$ and the last inequality is due to the Cauchy-Schwarz inequality. Since $f(x_{k})-f^*\geq 0$, by multiplying both sides by $k$ we obtain 
\begin{align}
0&\leq \frac{(1-\frac{\sigma}{2}\alpha_{k})k}{2\alpha_{k}}\Eset[\|x_{k}-x^*)\|^2 - \frac{k}{2\alpha_{k}}\Eset[\|x_{k+1}-x^*\|^2]\notag\\
&\qquad - \frac{\sigma k}{4}\Eset[\|x_{k}-x^*\|^2]   + (50M^2+2)k\alpha_{k;\tau(\alpha_{k})}\Eset\left[\|x_{k}-x^*\|^2\right] + 5(1+15M^2)k\alpha_{k;\tau(\alpha_{k})}.\label{sec_SGD_sc:thm_rate:eq1}
\end{align}
Recall that $\alpha_{k} = 4/\sigma k$ we have
\begin{align*}
&\frac{(1-\frac{\sigma}{2}\alpha_{k})k}{\alpha_{k}} = \frac{\sigma k(k-2)}{4} \leq \frac{\sigma (k-1)^2}{4},\qquad \forall k\geq 1.
\end{align*}
By \eqref{sec_SGD_sc:notation_K*} we have for all $k\geq \Kcal^*$
\begin{align*}
& - \frac{\sigma }{4}\Eset[\|x_{k}-x^*\|^2]   + (50M^2+2)\alpha_{k;\tau(\alpha_{k})}\Eset\left[\|x_{k}-x^*\|^2\right]  \leq 0.\\
&k\alpha_{k;\tau(\alpha_{k})}\leq k\tau(\alpha_{k})\alpha_{k-\tau(\alpha_{k})} = \frac{4k\tau(\alpha_{k})}{\sigma(k-\tau(\alpha_{k}))}\leq \frac{8\tau(\alpha_{k})}{\sigma}\cdot
\end{align*}
Thus, summing up both sides of \eqref{sec_SGD_sc:thm_rate:eq1} over $k$ and using the preceding three relations we have $\forall k\geq \Kcal^{*}$
\begin{align*}
0&\leq \frac{\sigma \Kcal^{*}(\Kcal^{*}-2)}{8} \Eset[\|x_{\Kcal^{*}}-x^*\|^2] - \frac{\sigma k^2}{8}\Eset[\|x_{k+1}-x^*\|^2] + \frac{40(15M^2+1)}{\sigma}
\sum_{t=\Kcal^{*}}^{k}\ln\left(\frac{\sigma t}{4}\right)\notag\\
&\leq \frac{\sigma \Kcal^{*}(\Kcal^{*}-2)}{8} \Eset[\|x_{\Kcal^{*}}-x^*\|^2] - \frac{\sigma k^2}{8}\Eset[\|x_{k+1}-x^*\|^2] + \frac{40(15M^2+1)\log(\sigma k/4)}{\sigma k},
\end{align*}
which implies \eqref{sec_SGD_sc:thm_rate:ineq}, i.e., for all $k\geq \Kcal^{*}$
\begin{align*}
\Eset[\|x_{k+1}-x^*\|^2]\leq \frac{\Kcal^{*}(\Kcal^{*}-2)}{k^2}\Eset[\|x_{\Kcal^{*}}-x^*\|^2] + \frac{320(15M^2+1)\log(\sigma k/4)}{k\sigma}\cdot
\end{align*}
\end{proof}

\subsection{Convergence analysis: smoothness and error bound conditions}\label{sec:SGD_eb}
In this section, we study the convergence rate of \eqref{sec_SGD:alg} when $f$ is a smooth convex function and satisfies the error bound condition. In particular, given a point $x$ we denote by $\hat{x}$ the projection of $x$ to $\Xcal^*$, i.e., 
\begin{align}
\hat{x} = \arg\min_{y\in\Xcal^*}\|x-y\|.     
\end{align}
The error bound condition is then defined as below \cite{KarimiNS2016}. 
\begin{assump}\label{assump:error_bound}
We assume that $\nabla f$ satisfies the following condition for all $x$
\begin{align}
\|\nabla f(x)\|_{*} \geq \sigma \|x-\hat{x}\|.\label{assump:error_bound:ineq}
\end{align}
\end{assump}
We note that this error bound condition is weaker than the strong convexity assumption in \eqref{assump_sc:ineq}, i.e., \eqref{assump_sc:ineq} implies \eqref{assump:error_bound:ineq}. This condition does not imply that $x^*$ is unique if it exists, i.e., if $\Xcal^*$ is nonempty then it may contains more than one element. However, the convergence rate of the classic gradient descent under the this condition is the same as the one under strong convexity, that is, they both achieve a linear convergence rate \cite{KarimiNS2016,BolteNPS2017,LuoT1993,Zhang2017,Schopfer2016}. We derive the convergence rate of {\sf SGD} under Assumption \ref{assump:error_bound} in this section, which says that $x_{k}$ converges to $x^*$ in expectation at a rate
\begin{align}
\Eset\left[\|x_{k}-\xhat_{k}\|^2\right]  \lesssim \Ocal\left(\frac{\Eset\left[\|x_{\Kcal^*}-\xhat_{\Kcal^*}\|^2\right]}{k^2} + \frac{\log(k)}{k}\right),  
\end{align}    
where $\Kcal^*$ is a positive interger defined later. Basically, this result says that $x_{k}$ converges to the optimal set $\Xcal^*$ in expectation. To achieve this result, we present some properties of \eqref{assump:error_bound:ineq} under the smooth condition of $f$ \cite{LuoT1993,BolteNPS2017,Zhang2017,KarimiNS2016,Schopfer2016}.
\begin{assump}\label{assump:smooth}
$f$ is continuously differentiable and there exists a constant $L>0$ such that 
\begin{align}
    \|\nabla f(x) - \nabla f(y)\|_{*} \leq L\|x-y\| ,\qquad \forall x,y\in\Rset^{d}.    \label{assump_smooth:ineq}
\end{align}
\end{assump} 
Under this assumption, the condition \eqref{assump:error_bound:ineq} implies the Polyak-Lojasiewicz ({\sf PL}) condition
\begin{align}
\frac{L}{2\sigma}\|\nabla f(x)\|_{*}^2 \geq f(x) - f^*,\label{assump:PL:ineq}
\end{align}
which can be derived from the following inequality using $\nabla f(\xhat) = 0$
\begin{align}
f(x) - f^* = f(x)  - f(\hat{x}) \leq \frac{L}{2}\|x-\hat{x}\|^2 \leq \frac{L}{2\sigma}\|\nabla f(x)\|^2.\label{assump:PL:ineq1} 
\end{align}
Second, the error bound condition also implies the following quadratic growth ({\sf QG}) condition \cite{KarimiNS2016,Schopfer2016}
\begin{align}
\frac{2\sigma}{L}\|x-\xhat\|^2\leq f(x) - f^*.    \label{assump:QG:ineq}
\end{align}
To derive this rate, we  consider the following two key lemmas. The first lemma is to provide an upper bound on the function value by using the Lipschitz continuity of the gradient of $f$.
\begin{lem}\label{sec_SGD_eb:lem_f_fk}
Suppose that Assumption \ref{assump:smooth} holds. Let $\{x_{k}\}$ be generated by \eqref{sec_SGD:alg}. Then we have
\begin{align}
f(x_{k+1}) &\leq f(x_{k}) -\alpha_{k}\left(1-\frac{L\alpha_{k}}{2}\right)\|\nabla f(x_k)\|_{*}^2 -\alpha_{k}\left(1-L\alpha_{k}\right)\langle\nabla f(x_k), G(x_{k};\xi_{k})-\nabla f(x_k)\rangle\notag\\    
&\qquad + 4M^2L\alpha_{k}^2\|x_{k}\|^2 + 4M^2L\alpha_{k}^2.\label{sec_SGD_eb:lem_f_fk:Ineq}
\end{align}
\end{lem}

\begin{proof}
By \eqref{assump_smooth:ineq} we have
\begin{align}
f(x_{k+1}) &\leq f(x_k) + \langle\nabla f(x_k), x_{k+1} -x_{k}\rangle + \frac{L}{2}\|x_{k+1}-x_{k}\|^2\notag\\
&= f(x_k) - \alpha_{k}\langle\nabla f(x_k), G(x_{k};\xi_{k})\rangle + \frac{L\alpha_{k}^2}{2}\|G(x_{k};\xi_{k})\|_{*}^2\notag\\
&= f(x_k) - \alpha_{k}\|\nabla f(x_k)\|_{*}^2 - \alpha_{k}\langle\nabla f(x_k), G(x_{k};\xi_{k})-\nabla f(x_k)\rangle\notag\\ 
&\qquad + \frac{L\alpha_{k}^2}{2}\|\nabla f(x_{k})\|_{*}^2 + L\alpha_{k}^2\langle\nabla f(x_k), G(x_{k};\xi_{k})-\nabla f(x_k)\rangle + \frac{L\alpha_{k}^2}{2}\|G(x_{k};\xi_{k})-\nabla f(x_{k})\|_{*}^2 \notag\\
&= f(x_{k}) -\alpha_{k}\left(1-\frac{L\alpha_{k}}{2}\right)\|\nabla f(x_k)\|_{*}^2 -\alpha_{k}\left(1-L\alpha_{k}\right)\langle\nabla f(x_k), G(x_{k};\xi_{k})-\nabla f(x_k)\rangle\notag\\ 
&\qquad + \frac{L\alpha_{k}^2}{2}\|G(x_{k};\xi_{k})-\nabla f(x_{k})\|_{*}^2. \label{sec_SGD_eb:lem_f_fk:eq1}
\end{align}
By \eqref{assump:Lipschitz_sample:ineq} we have
\begin{align*}
\|G(x_{k};\xi_{k})-\nabla f(x_{k})\|_{*}^2 &\leq 2(\|G(x_{k};\xi_{k})\|_{*}^2 + \|\nabla f(x_{k})\|_{*}^2)\leq 8M^2(\|x_{k}\|^2+1).
\end{align*}
which when substituting into \eqref{sec_SGD_eb:lem_f_fk:eq1} yields \eqref{sec_SGD_eb:lem_f_fk:Ineq}.
\end{proof}
We next analyse the third term on the right-hand side of \eqref{sec_SGD_eb:lem_f_fk:Ineq}.
\begin{lem}\label{sec_SGD_eb:lem_bias}
Suppose that Assumptions \ref{assump:stationary}, \ref{assump:mix}, and \ref{assump:Lipschitz_sample} hold. Let $\{x_{k}\}$ be generated by \eqref{sec_SGD:alg}. Then  $\forall x\in\Xcal$
\begin{align}
\Eset\left[-\langle\nabla f(x_k), G(x_{k};\xi_{k})-\nabla f(x_k)\rangle\right]\leq M(24ML\alpha_{k;\tau(\alpha_{k})}+\alpha_{k})\Eset\left[\|x_{k}\|^2\right] + 4M(\alpha_{k}+6ML\alpha_{k;\tau(\alpha_{k})}). \label{sec_SGD_eb:lem_bias:ineq}
\end{align}
\end{lem}

\begin{proof}
Consider
\begin{align}
&-\langle\nabla f(x_k), G(x_{k};\xi_{k})-\nabla f(x_k)\rangle\notag\\
&= \langle\nabla f(x_k) - \nabla f(x_{k-\tau(\alpha_{k})}), \nabla f(x_k) - G(x_{k};\xi_{k})\rangle + \langle\nabla f(x_{k-\tau(\alpha_{k})}), \nabla f(x_k) - G(x_{k};\xi_{k})\rangle.\label{sec_SGD_eb:lem_bias:eq1}
\end{align}
First, using \eqref{assump:mix} we have
\begin{align}
&\Eset\left[\langle\nabla f(x_{k-\tau(\alpha_{k})}), \nabla f(x_k) - G(x_{k};\xi_{k})\rangle\,|\,\Fcal_{k-\tau(\alpha_{k})}\right]\notag\\
&=\langle\nabla f(x_{k-\tau(\alpha_{k})}), \Eset\left[\nabla f(x_k) - G(x_{k};\xi_{k})\rangle\,|\,\Fcal_{k-\tau(\alpha_{k})}\right] \leq \alpha_{k}\|\nabla f(x_{k-\tau(\alpha_{k})})\|\notag\\
&\stackrel{\eqref{assump:Lipschitz_sample:ineq}}{\leq} M\alpha_{k}(\|x_{k-\tau(\alpha_{k})}\| + 1)\stackrel{\eqref{sec_SGD_sc:lem_xk_bounded_ineq2}}{\leq} 2M\alpha_{k}\|x_{k}\| + 3M\alpha_{k}\leq M\alpha_{k}\|x_{k}\|^2 + 4M\alpha_{k}.\label{sec_SGD_eb:lem_bias:eq1a}
\end{align}
Second, using \eqref{assump:Lipschitz_sample:ineq}, \eqref{sec_SGD_sc:lem_xk_bounded_ineq2}, \eqref{assump_smooth:ineq}, and \eqref{assump_smooth:ineq} we obtain
\begin{align}
&\langle\nabla f(x_k) - \nabla f(x_{k-\tau(\alpha_{k})}), \nabla f(x_k) - G(x_{k};\xi_{k})\rangle \leq L\|x_{k}-x_{k-\tau(\alpha_{k})}\|\left(\|\nabla f(x_k)\| + \|G(x_{k};\xi_{k})\|\right)\notag\\
&\stackrel{\eqref{sec_SGD_sc:lem_xk_bounded_ineq2}}{\leq} 2ML(\|x_{k}\| + 1)\left(6M\alpha_{k;\tau(\alpha_{k})}\|x_{k}\| + 6M\alpha_{k;\tau(\alpha_{k})}\right)\leq 24M^2L\alpha_{k;\tau(\alpha_{k})}\|x_{k}\|^2 + 24M^2L\alpha_{k;\tau(\alpha_{k})}.\label{sec_SGD_eb:lem_bias:eq1b}
\end{align}
Taking the expectation on both sides of \eqref{sec_SGD_eb:lem_bias:eq1} and using \eqref{sec_SGD_eb:lem_bias:eq1a} and \eqref{sec_SGD_eb:lem_bias:eq1b} yields \eqref{sec_SGD_eb:lem_bias:ineq}.
\end{proof}
We now present the main result in this section, which is the convergence rate of ${x_{k}}$ in expectation under Assumption \ref{assump:error_bound}.  First, we consider the choice of step sizes $\alpha_{k}$ as
\begin{align}
\alpha_{k} = \frac{\alpha_{0}}{k},\qquad \text{where }     \alpha_{0} = \min\left\{\frac{1}{2L};\,\frac{2L}{\sigma}\right\}.\label{sec_SGD_eb:stepsizes}  
\end{align}
Moreover, with some abuse of notation, we consider in this section  $\Kcal^*$ satisfying
\begin{align}
\left\{\begin{array}{l}
\tau(\alpha_{k})\alpha_{k-\tau(\alpha_{k})}\leq \max\left\{\frac{\log(2)}{M},\;\frac{\sigma}{L\sigma + 4M + 104M^2L}\right\}\\
\ln\left(\frac{k}{\alpha_{0}}\right)\leq \frac{k}{2}
\end{array}\right.\qquad \forall k \geq \Kcal^*,   \label{sec_SGD_eb:notation_K*}
\end{align} 
which also satisfies \eqref{sec_SGD:notation_K*}. The main result in this section is stated in the following theorem.

\begin{thm}\label{sec_SGD_eb:thm}
Suppose that Assumptions \ref{assump:optimal_set}--\ref{assump:Lipschitz_sample}, \ref{assump:error_bound}, \ref{assump:smooth} hold. Let $\{x_{k}\}$ be generated by \eqref{sec_SGD:alg}. In addition, let $\alpha_{k}$ satisfy \eqref{sec_SGD_eb:stepsizes} and  $\Kcal^*$ be defined in \eqref{sec_SGD_eb:notation_K*}. We denote by $D_{1}$ and $D_{2}$ two constants as
\begin{align}
    \begin{aligned}
D_{1} &= 4M^2L + 4M + (2M + 8M^2L)\max_{x\in\Xcal^*}\|x\|^2\\
D_{2} &= 24M^2L + 48M^2L\max_{x\in\Xcal^*}\|x\|^2.  
    \end{aligned}\label{sec_SGD_eb:thm:constants}
\end{align}
Then we have for all $k\geq\Kcal^*$
\begin{align}
\Eset\left[\|x_{k+1} - \xhat_{k+1}\|^2\right] \leq \frac{L^2(\Kcal^*-1)^2\Eset\left[\|x_{\Kcal^*}-\xhat_{\Kcal^*}\|^2\right]}{4\sigma k^2}   + \frac{L\alpha_{0}^2D_{1}}{2\sigma k} + \frac{L\alpha_{0}D_{2}\ln\left(\frac{k}{\alpha_{0}}\right)}{\sigma k},\label{sec_SGD_eb:thm:ineq}
\end{align}
where recall that $\xhat_{k}$ is the projection of $x_{k}$ to $\Xcal^*$.
\end{thm}

\begin{proof} 
By \eqref{sec_SGD_eb:stepsizes} we have $(1-L\alpha_{k})\in(0,1)$. Taking the expectation on both sides of \eqref{sec_SGD_eb:lem_f_fk:Ineq} and using \eqref{sec_SGD_eb:lem_bias:ineq} we obtain 
\begin{align*}
\Eset\left[f(x_{k+1})\right] &\leq \Eset\left[f(x_{k})\right] -\alpha_{k}\left(1-\frac{L\alpha_{k}}{2}\right)\Eset\left[\|\nabla f(x_k)\|_{*}^2\right] +M\alpha_{k}\left(1-L\alpha_{k}\right)(24ML\alpha_{k;\tau(\alpha_{k})}+\alpha_{k})\Eset\left[\|x_{k}\|^2\right] \notag\\    
&\qquad + 4M\alpha_{k}\left(1-L\alpha_{k}\right)(\alpha_{k}+6ML\alpha_{k;\tau(\alpha_{k})}) + 4M^2L\alpha_{k}^2\|x_{k}\|^2 + 4M^2L\alpha_{k}^2\notag\\
&\stackrel{\eqref{assump:error_bound:ineq}}{\leq} \Eset\left[f(x_{k})\right] -\sigma\alpha_{k}\left(1-\frac{L\alpha_{k}}{2}\right)\Eset\left[\|x_{k}-\xhat_{k}\|^2\right] + 4M(ML+1)\alpha_{k}^2 + 24M^2L\alpha_{k;\tau(\alpha_{k})}\alpha_{k} \notag\\   
&\qquad + M(1+4ML)\alpha_{k}^2 \Eset\left[\|x_{k}\|^2\right] + 24M^2L\alpha_{k;\tau(\alpha_{k})}\alpha_{k}\Eset\left[\|x_{k}\|^2\right],
\end{align*}
which by using the Cauchy-Schwarz inequality $\|x+y\|^2 \leq 2\|x\|^2 + 2\|y\|^2$ gives
\begin{align}
\Eset\left[f(x_{k+1})\right]
&\leq \Eset\left[f(x_{k})\right] -\sigma\alpha_{k}\left(1-\frac{L\alpha_{k}}{2}\right)\Eset\left[\|x_{k}-\xhat_{k}\|^2\right] + 4M(ML+1)\alpha_{k}^2 + 24M^2L\alpha_{k;\tau(\alpha_{k})}\alpha_{k} \notag\\   
&\qquad + 2M(1+4ML)\alpha_{k}^2 \Eset\left[\|x_{k}-\xhat_{k}\|^2\right] + 48M^2L\alpha_{k;\tau(\alpha_{k})}\alpha_{k}\Eset\left[\|x_{k}-\xhat_{k}\|^2\right]\notag\\   
&\qquad + 2M(1+4ML)\alpha_{k}^2 \Eset\left[\|\xhat_{k}\|^2\right] + 48M^2L\alpha_{k;\tau(\alpha_{k})}\alpha_{k}\Eset\left[\|\xhat_{k}\|^2\right]\notag\\
&\leq \Eset\left[f(x_{k})\right] -\sigma\alpha_{k}\left(1-\frac{L\alpha_{k}}{2}\right)\Eset\left[\|x_{k}-\xhat_{k}\|^2\right] + D_{1}\alpha_{k}^2 + D_{2}\alpha_{k;\tau(\alpha_{k})}\alpha_{k} \notag\\   
&\qquad + 2M(1+4ML)\alpha_{k}^2 \Eset\left[\|x_{k}-\xhat_{k}\|^2\right] + 48M^2L\alpha_{k;\tau(\alpha_{k})}\alpha_{k}\Eset\left[\|x_{k}-\xhat_{k}\|^2\right].\label{sec_SGD_eb:thm:eq1}
\end{align}
where $D_{1}$ and $D_{2}$ are defined in \eqref{sec_SGD_eb:thm:constants}. By \eqref{sec_SGD_eb:notation_K*} we have for all $k\geq\Kcal^*$
\begin{align*}
\left(\frac{L\sigma}{2} + 2M(1+4ML)\right)\alpha_{k} + 48M^2L \tau(\alpha_{k})\alpha_{k-\tau(\alpha_{k})}\leq \frac{L\sigma + 4M + 104M^2L}{2}\tau(\alpha_{k})\alpha_{k-\tau(\alpha_{k})} \leq \frac{\sigma}{2}.    
\end{align*}
Then by \eqref{sec_SGD_eb:thm:eq1} we have for all $k\geq\Kcal^*$
\begin{align*}
\Eset\left[f(x_{k+1})\right]
&\leq \Eset\left[f(x_{k})\right] -\frac{\sigma}{2}\alpha_{k}\Eset\left[\|x_{k}-\xhat_{k}\|^2\right] + D_{1}\alpha_{k}^2 + D_{2}\alpha_{k;\tau(\alpha_{k})}\alpha_{k},
\end{align*}
which by substracting both sides by $f^*$ and using \eqref{assump:PL:ineq} yields
\begin{align}
\Eset\left[f(x_{k+1})\right] - f^*
&\leq \Eset\left[f(x_{k})\right] - f^* -\frac{\sigma}{L}\alpha_{k}\left(\Eset\left[f(x_{k})\right] - f^*\right) + D_{1}\alpha_{k}^2 + D_{2}\alpha_{k;\tau(\alpha_{k})}\alpha_{k}\notag\\
&= \left(1-\frac{\sigma}{L}\alpha_{k}\right)\left(\Eset\left[f(x_{k})\right] - f^*\right) + D_{1}\alpha_{k}^2 + D_{2}\alpha_{k;\tau(\alpha_{k})}\alpha_{k}.\label{sec_SGD_eb:thm:eq1a}
\end{align}
Using the choice of step sizes in \eqref{sec_SGD_eb:stepsizes} we have
\begin{align*}
k^2 \left(1-\frac{\sigma}{L}\alpha_{k}\right) = k^2 \left(1-\frac{\sigma}{L}\frac{\alpha_{0}}{k}\right) \leq  k^2\left(1-\frac{2}{k}\right)\leq (k-1)^2.
\end{align*}
In addition, by \eqref{sec_SGD_eb:notation_K*} we also have
\begin{align*}
k\alpha_{k;\tau(\alpha_{k})} \leq \frac{\alpha_{0}k\tau(\alpha_{k})}{k-\tau(\alpha_{k})}\leq 2\alpha_{0}\tau(\alpha_{k}) = 2\alpha_{0}\ln\left(\frac{k}{\alpha_{0}}\right)     
\end{align*}
Thus, by multiplying both sides of \eqref{sec_SGD_eb:thm:eq1} by $k^2$ and using the preceding two relations we obtain for all $k\geq\Kcal^*$
\begin{align*}
k^2\left(\Eset\left[f(x_{k+1})\right] - f^*\right) &\leq (k-1)^2 \left(\Eset\left[f(x_{k})\right] - f^*\right) + \alpha_{0}^2D_{1} + 2\alpha_{0}D_{2}\ln\left(\frac{\alpha_{0}}{k}\right)\notag\\
&\leq (\Kcal^*-1)^2 \left(\Eset\left[f(x_{\Kcal^*})\right]- f^*\right)  + (k-\Kcal^*) \alpha_{0}^2D_{1}  + 2\alpha_{0}D_{2}\sum_{t=\Kcal^*}^{k-1}\ln\left(\frac{t}{\alpha_{0}}\right)\notag\\
&\stackrel{\eqref{assump:PL:ineq1}}{\leq} \frac{L(\Kcal^*-1)^2\Eset\left[\|x_{\Kcal^*}-\xhat_{\Kcal^*}\|^2\right]}{2}   + k\alpha_{0}^2D_{1} + 2\alpha_{0}D_{2}k\ln\left(\frac{k}{\alpha_{0}}\right),
\end{align*}
which by using \eqref{assump:QG:ineq} and multiplying both sides by $L/2\sigma k^2$ gives us \eqref{sec_SGD_eb:thm:ineq}.
\end{proof}

\subsection{Convergence analysis: smoothness and nonconvexity}\label{sec:SGD_nc}
In this section, we study the convergence rate of \eqref{sec_SGD:alg} when $f$ is a smooth nonconvex function, i.e., $f$ satisfies Assumption \ref{assump:smooth}. In particular, we study the convergence to a stationary point of \eqref{sec_optimization:prob} in expectation, which happens at a rate
\begin{align*}
\Eset\left[\|\nabla f(x_R)\|_{*}^2\right] 
\leq \Ocal\left(\frac{\Eset\left[\|x_{\Kcal^*} - \xhat_{\Kcal^*}\|^2\right]}{\sqrt{T}}\right) + \Ocal\left(\frac{\ln^2(T)}{\sqrt{T}}\right),  
\end{align*}
where $T$ is some predetermined positive integer, and $x_{R}$ is randomly selected from the sequence $\{x_{k}\}$, for $k\in[\Kcal^*,T]$,  with some probability. This random rule of choosing the return point $x_{R}$ is adopted from \cite{GhadimiL2013} for nonconvex optimization problems.

We now proceed to show our main result in this section. Given a positive integer $T$, consider the following choice of step sizes
\begin{align}
\alpha_{k} = \frac{\alpha_{0}}{\sqrt{k}},\qquad \text{for all }k\in[1,T],     \label{sec_SGD_nc:stepsizes}
\end{align}
and $\alpha_{0}$ satisfies
\begin{align}
0< \alpha_{0} \leq \max\left\{\max_{k\in[1,T]}\left\{ \frac{\sqrt{k+1}-\sqrt{k}}{2M}\right\}; \frac{1}{L}\right\}\cdot     \label{sec_SGD_nc:stepsizes:alpha0}
\end{align}
Moreover, we denote by $\delta$ 
\begin{align}
    \delta = \max_{k\in[1,T]} \left\{\frac{\sqrt{k}+M\alpha_{0}}{\sqrt{k+1}}\right\}  \in (0,1).\label{sec_SGD_nc:notation_delta}
\end{align}
Finally, let $\Kcal^*$ be a positive integer such that
\begin{align}
\tau(\alpha_{k})\alpha_{k-\tau(\alpha_{k})}\leq \frac{\log(2)}{M}\quad \text{and }\quad \ln\left(\frac{\sqrt{k}}{\alpha_{0}}\right)\leq \min\left\{\delta^{k/2},\;\frac{\alpha_{0}}{\sqrt{k}}\right\},\quad \forall k \geq \Kcal^*.     \label{sec_SGD_nc:notation_K*}
\end{align}
Under these conditions, we next consider the following lemma, which provide upper bounds for the time $\alpha$-weighted average of the iterates generated by \eqref{sec_SGD:alg}. 
\begin{lem}\label{sec_SGD_nc:lem_xk_bound}
Suppose that Assumption \ref{assump:Lipschitz_sample} holds. Let $\{x_{k}\}$ be generated by \eqref{sec_SGD:alg}. Given a positive integer $T\geq \Kcal^*$, let $\alpha_{k}$ be chosen in \eqref{sec_SGD_nc:stepsizes}. Then we have for all $k\in [\Kcal^*,T]$
\begin{align}
&\sum_{k=\Kcal^*}^{T}\alpha_{k}^2\|x_{k}\|^2 \leq D_{3} \triangleq  \frac{3\alpha_{0}^2}{1-\delta^2} + \frac{3M^2\alpha_{0}^4}{(1-\delta)^3} + \frac{24 M^2\alpha_{0}^2}{(1-\delta)^2}\cdot \label{sec_SGD_nc:lem_xk_bound:ineq1}\\
&\sum_{k=\Kcal^*}^{T}  \alpha_{k;\tau(\alpha_{k})}\alpha_{k}\|x_{k}\|^{2}\leq D_{4}\triangleq  \frac{3\sqrt{2}\alpha_{0}^2}{1-\delta} + \frac{3\sqrt{2}M^2\alpha_{0}^4}{(1-\delta)(1-\sqrt{\delta})} + \frac{24\sqrt{2} M^2\alpha_{0}^2}{(1-\delta)^2}\cdot\label{sec_SGD_nc:lem_xk_bound:ineq2}
\end{align}
\end{lem}

\begin{proof}
Using \eqref{sec_SGD_sc:lem_xk_bounded_Eq1a}, $\alpha_{k} = \alpha_{0}/\sqrt{k}$, and \eqref{sec_SGD_nc:notation_delta} we have for all $k\geq [\Kcal^*,T]$
\begin{align*}
\alpha_{k+1}\|x_{k+1}\| &\leq \alpha_{k+1}(1+M\alpha_{k})\|x_{k}\| + M\alpha_{k}\alpha_{k+1} = \left(\frac{\sqrt{k}+M\alpha_{0}}{\sqrt{k+1}}\right)\alpha_{k}\|x_{k}\| + M\alpha_{k}\alpha_{k+1}\notag\\
&\leq \delta\alpha_{k}\|x_{k}\| + M\alpha_{k}\alpha_{k+1} \leq \delta^{k}\alpha_{0}\|x_{1}\| + M\sum_{t=1}^{k}\alpha_{k}^2\delta^{k-t}\notag\\
&= \delta^{k}\alpha_{0}\|x_{1}\| + M\sum_{t=1}^{\lfloor k/2 \rfloor}\alpha_{k}^2\delta^{k-t} + M\sum_{t=\lceil k/2\rceil}^{k}\alpha_{k}^2\delta^{k-t}\notag\\
&\leq \delta^{k}\alpha_{0}\|x_{1}\| + M\alpha_{0}^2\sum_{t=1}^{\lfloor k/2 \rfloor}\delta^{k-t} + M\alpha_{\lceil k/2\rceil}^2\sum_{t=\lceil k/2\rceil}^{k}\delta^{k-t}\notag\\
&\leq \delta^{k}\alpha_{0}\|x_{1}\| + \frac{M\alpha_{0}^2}{1-\delta}\delta^{\lceil k/2\rceil} + \frac{M}{1-\delta}\alpha_{\lceil k/2\rceil}^2,
\end{align*}
which using the Cauchy-Schwarz inequality gives
\begin{align*}
\alpha_{k+1}^2\|x_{k+1}\|^{2}\leq 3\alpha_{0}^2\|x_{1}\|^2\delta^{2k} + \frac{3M^2\alpha_{0}^{4}}{(1-\delta)^2}\delta^{k+1} + \frac{12 M^2\alpha_{0}^2}{(1-\delta)^2}\frac{1}{k^2}\cdot
\end{align*}
Summing up both sides from $k = \Kcal^*,\ldots,T$ yields \eqref{sec_SGD_nc:lem_xk_bound:ineq1} 
\begin{align*}
\sum_{k=\Kcal^*}^{T}\alpha_{k}^2\|x_{k}\|^2 &\leq   \sum_{k=\Kcal^*}^{T} 3\alpha_{0}^2\|x_{1}\|^2\delta^{2k} + \sum_{k=\Kcal^*}^{T} \frac{3M^2\alpha_{0}^{4}}{(1-\delta)^2}\delta^{k+1}  + \sum_{k=\Kcal^*}^{T} \frac{12 M^2\alpha_{0}^2}{(1-\delta)^2}\frac{1}{k^2} \notag\\
&\leq \frac{3\alpha_{0}^2}{1-\delta^2} + \frac{3M^2\alpha_{0}^4}{(1-\delta)^3} + \frac{24 M^2\alpha_{0}^2}{(1-\delta)^2}\cdot
\end{align*}
Similarly, we consider
\begin{align*}
&\alpha_{k+1;\tau(\alpha_{k+1})}\alpha_{k+1}\|x_{k+1}\|^{2}\leq \frac{\alpha_{k+1;\tau(\alpha_{k+1})}}{\alpha_{k+1}}\alpha_{k+1}^2\|x_{k+1}\|^{2}\leq \frac{\tau(\alpha_{k+1})\alpha_{k+1-\tau(\alpha_{k+1})}}{\alpha_{k+1}}\alpha_{k+1}^2\|x_{k+1}\|^{2}\notag\\
&= \frac{\tau(\alpha_{k+1})\sqrt{k+1}}{\sqrt{k+1-\tau(\alpha_{k+1})}}\alpha_{k+1}^2\|x_{k+1}\|^{2}\leq \sqrt{2}\tau(\alpha_{k+1})\alpha_{k+1}^2\|x_{k+1}\|^{2},
\end{align*}
which by using \eqref{sec_SGD_nc:notation_K*} gives \eqref{sec_SGD_nc:lem_xk_bound:ineq2}
\begin{align*}
&\sum_{k=\Kcal^*}^{T}  \alpha_{k;\tau(\alpha_{k})}\alpha_{k}\|x_{k}\|^{2}\leq  \sqrt{2}\sum_{k=\Kcal^*}^{T}  \tau(\alpha_{k})\alpha_{k}^2\|x_{k}\|^{2}\notag\\
&\leq \sum_{k=\Kcal^*}^{T} 3\sqrt{2}\alpha_{0}^2\|x_{1}\|^2\delta^{k} + \sum_{k=\Kcal^*}^{T} \frac{3M^2\alpha_{0}^{4}}{(1-\delta)^2}\delta^{k/2}  + \sum_{k=\Kcal^*}^{T} \frac{12 M^2\alpha_{0}^2}{(1-\delta)^2}\frac{1}{k^{3/2}}\notag\\
&\leq  \frac{3\sqrt{2}\alpha_{0}^2}{1-\delta} + \frac{3\sqrt{2}M^2\alpha_{0}^4}{(1-\delta)(1-\sqrt{\delta})} + \frac{24\sqrt{2} M^2\alpha_{0}^2}{(1-\delta)^2}\cdot
\end{align*}

\end{proof}

We now state the main result in this section, which is the rate of {\sf SGD} when $f$ is nonconvex and smooth. To do it, we adopt the randomized stopping rule in \cite{GhadimiL2013}, that is, given a sequence $\{x_{k}\}$ generated by {\sf SGD} we derive the convergence of $x_{R}$, a random point selected from this sequence. More detail is given in the following theorem.  
\begin{thm}\label{sec_SGD_nc:thm_rate}
Suppose that Assumptions \ref{assump:stationary}-- \ref{assump:Lipschitz_sample}, and \ref{assump:smooth} hold. Let $\{x_{k}\}$ be generated by \eqref{sec_SGD:alg}. Given a positive integer $T\geq \Kcal^*$, let $\alpha_{k}$ satisfies \eqref{sec_SGD_nc:stepsizes}. 
In addition, we randomly select $x_{R}$ from the sequence $\{x_{k}\}$, for $k\in[\Kcal^*,T]$, with probability $P_{R}(\cdot)$ defined as
\begin{align}
P_{R}(k) = \Pset(R=k) = \frac{2\alpha_{k}-L\alpha_{k}^2}{\sum_{k=\Kcal^*}^{T}(2\alpha_{k}-L\alpha_{k}^2)}\cdot     \label{sec_SGD_nc:thm_rate:P_R}
\end{align}
Then we have 
\begin{align}
\Eset\left[\|\nabla f(x_R)\|_{*}^2\right] 
&\leq \frac{2L\Eset\left[\|x_{\Kcal^*} - \xhat_{\Kcal^*}\|^2\right] + 4D_{3}M(4ML+1) + 96D_{4}M^2L}{3\alpha_{0}\sqrt{T}}\notag\\ 
&\qquad  + \frac{ 4\alpha_{0}^2(1+\ln(T)) + 2\ln^2\left(\frac{T}{\alpha_{0}}\right)}{3\alpha_{0}\sqrt{T}},\label{sec_SGD_nc:thm_rate:ineq}    
\end{align}
where recall that $\xhat$ is the projection of $x$ to the optimal set $\Xcal^*$.
\end{thm}

\begin{proof}
Taking the expectation of  \eqref{sec_SGD_eb:lem_f_fk:Ineq} and using \eqref{sec_SGD_eb:lem_bias:ineq} we obtain
\begin{align*}
\Eset\left[f(x_{k+1})\right] &\leq \Eset\left[f(x_{k})\right] -\alpha_{k}\left(1-\frac{L\alpha_{k}}{2}\right)\Eset\left[\|\nabla f(x_k)\|_{*}^2\right] + 4M^2L\alpha_{k}^2\Eset\left[\|x_{k}\|^2\right] + 4M^2L\alpha_{k}^2\notag\\
&\qquad + M\alpha_{k}\left(1-L\alpha_{k}\right)(24ML\alpha_{k;\tau(\alpha_{k})}+\alpha_{k})\Eset\left[\|x_{k}\|^2\right]\notag\\ 
&\qquad + 4M\alpha_{k}\left(1-L\alpha_{k}\right)(\alpha_{k}+6ML\alpha_{k;\tau(\alpha_{k})})\notag\\
&\leq \Eset\left[f(x_{k})\right] -\alpha_{k}\left(1-\frac{L\alpha_{k}}{2}\right)\Eset\left[\|\nabla f(x_k)\|_{*}^2\right]  + 4M(ML+1)\alpha_{k}^2 + 24M^2L\alpha_{k;\tau(\alpha_{k})}\alpha_{k} \notag\\
&\qquad + M(4ML+1)\alpha_{k}^2\Eset\left[\|x_{k}\|^2\right] + 24M^2L\alpha_{k;\tau(\alpha_{k})}\alpha_{k}\Eset\left[\|x_{k}\|^2\right],
\end{align*}
which by summing up both sides over $k = \Kcal^*,\ldots, T$ and reorganizing we obtain \begin{align}
&\sum_{k = \Kcal^*}^{T}   \alpha_{k}\left(1-\frac{L\alpha_{k}}{2}\right)\Eset\left[\|\nabla f(x_k)\|_{*}^2\right]\notag\\
&\leq \Eset\left[f(x_{\Kcal^*}) - f(x_{T+1})\right]  + \sum_{k=\Kcal^*}^{T}\left(4M(ML+1)\alpha_{k}^2 + 24M^2L\alpha_{k;\tau(\alpha_{k})}\alpha_{k}\right)\notag\\
&\qquad + \sum_{k=\Kcal^*}^{T}\left[M(4ML+1)\alpha_{k}^2\Eset\left[\|x_{k}\|^2\right] + 24M^2L\alpha_{k;\tau(\alpha_{k})}\alpha_{k}\Eset\left[\|x_{k}\|^2\right]\right]\notag\\
&\leq \Eset\left[f(x_{\Kcal^*}) - f^*\right]  + \sum_{k=\Kcal^*}^{T}\left(4M(ML+1)\alpha_{k}^2 + 24M^2L\alpha_{k;\tau(\alpha_{k})}\alpha_{k}\right)\notag\\
&\qquad + D_{3}M(4ML+1) + 24D_{4}M^2L,\label{sec_SGD_nc:thm_rate:Eq1}
\end{align}
where $D_{3}$ and $D_{4}$ are given in \eqref{sec_SGD_nc:lem_xk_bound:ineq1} and \eqref{sec_SGD_nc:lem_xk_bound:ineq2}, respectively. Note that, by \eqref{sec_SGD_nc:stepsizes} we have
\begin{align*}
&\sum_{k = \Kcal^*}^{T}   \alpha_{k}\left(1-\frac{L\alpha_{k}}{2}\right)\geq \frac{3}{4}\sum_{k = \Kcal^*}^{T}   \alpha_{k} = \frac{3}{4}\sum_{k = \Kcal^*}^{T}\frac{\alpha_{0}}{\sqrt{k}} \geq \frac{3\alpha_{0}\sqrt{T}}{4}\cdot\notag\\
&\sum_{k=\Kcal^*}^{T}\alpha_{k}^2 = \sum_{k=\Kcal^*}^{T}\frac{\alpha_{0}^2}{k}\leq \alpha_{0}^2(1+\ln(T)).\notag\\
&\sum_{k=\Kcal^*}^{T}\alpha_{k;\tau(\alpha_{k})}\alpha_{k}\leq \sum_{k=\Kcal^*}^{T}\frac{\alpha_{0}^2\ln(\tau(\alpha_{k}))}{\sqrt{k}\sqrt{k-\ln(\tau(\alpha_{k}))}}\leq \sum_{k=\Kcal^*}^{T}\frac{2\alpha_{0}^2\ln(\tau(\alpha_{k}))}{k}\leq \frac{\ln^2\left(\frac{T}{\alpha_{0}}\right)}{2}\cdot
\end{align*}
Diving both sides of \eqref{sec_SGD_nc:thm_rate:Eq1} by $\sum_{k = \Kcal^*}^{T}   \alpha_{k}\left(1-\frac{L\alpha_{k}}{2}\right)$ and using the preceding three relations yield 
\begin{align*}
&\frac{\sum_{k = \Kcal^*}^{T}  \alpha_{k}\left(1-\frac{L\alpha_{k}}{2}\right)\Eset\left[\|\nabla f(x_k)\|_{*}^2\right]}{\sum_{k = \Kcal^*}^{T}   \alpha_{k}\left(1-\frac{L\alpha_{k}}{2}\right)}\notag\\ 
&\leq \frac{4\Eset\left[f(x_{\Kcal^*}) - f^*\right] + 4D_{3}M(4ML+1) + 96D_{4}M^2L}{3\alpha_{0}\sqrt{T}} + \frac{ 4\alpha_{0}^2(1+\ln(T)) + 2\ln^2\left(\frac{T}{\alpha_{0}}\right)}{3\alpha_{0}\sqrt{T}}\notag\\
&\leq \frac{2L\Eset\left[\|x_{\Kcal^*} - \xhat_{\Kcal^*}\|^2\right] + 4D_{3}M(4ML+1) + 96D_{4}M^2L}{3\alpha_{0}\sqrt{T}} + \frac{ 4\alpha_{0}^2(1+\ln(T)) + 2\ln^2\left(\frac{T}{\alpha_{0}}\right)}{3\alpha_{0}\sqrt{T}},    
\end{align*}
where the last inequlaity is due to the Lipschitz continuity of $\nabla f$. Using the definition of $x_{R}$ in \eqref{sec_SGD_nc:thm_rate:P_R} immediately yields \eqref{sec_SGD_nc:thm_rate:ineq}.
\end{proof}




\bibliographystyle{IEEEtran}
\bibliography{refs}
\end{document}